\documentclass[12pt]{amsart}
\usepackage[english]{babel}
\usepackage{amsfonts,amssymb,latexsym,amscd}

\theoremstyle{plain}
\newtheorem{theorem}{Theorem}

\newtheorem{proposition}{Proposition}
\newtheorem{corollary}{Corollary}
\newtheorem{problem}{Problem}
\theoremstyle{definition}
\newtheorem{definition}{Definition}
\theoremstyle{remark}

\newtheorem{example}{Example}

\begin{document} 
 
\author{Pavel S. Gevorgyan}
\address{Moscow Pedagogical State University}
\email{pgev@yandex.ru}

\author{ A.\,A.~Nazaryan}
\address{Yerevan State University}
\email{aram1nazaryan@gmail.com}

\title{On Orbits and Bi-invariant Subsets of Binary $G$-Spaces}

\begin{abstract}
Orbits and bi-invariant subsets of binary $G$-spaces are studied. The problem of the distributivity of a binary action of a group $G$ on a space $X$, which was posed in 2016 by one of the authors, is solved. 
\end{abstract}


\maketitle

\section{Introduction}
The notions of a binary action of a group $G$ on a topological space $X$ and of a binary $G$-space were introduced in \cite{Gev}. The group $H_2(X)$ of all invertible continuous binary operations on a space $X$ acts binarily on the space $X$. Moreover, if a group $G$ acts binarily and effectively on $X$, then $G$ is a subgroup of $H_2(X)$ \cite{Gev2}. The category $G$-$Top^2$ of binary $G$-spaces and bi-equivariant mappings is a natural extension of the category $G$-$Top$ of all $G$-spaces and equivariant mappings, which, in turn, is an extension of the category $Top$; i.e.,
$$Top\subset G\text{-}Top\subset G\text{-}Top^2.$$ 

In \cite{Gev2}, the notion of a distributive binary $G$-space $X$ was introduced. One of the reasons why this notion is important is the special role played by distributive subgroups of the group $H_2(X)$ of all invertible continuous binary operations on $X$. For example, any topological group is a distributive subgroup of the group of invertible binary operations of some space \cite{Gev-izv}. This statement is the binary topological counterpart of Cayley’s classical theorem on the representation of any finite group by unary operations (permutations).

This paper is concerned with orbits and bi-invariant sets in binary $G$-spaces. We emphasize that transferring the basic notions of the theory of $G$-spaces to the theory of binary $G$-spaces and studying them is not always easy. Substantial differences arise already in considering bi-invariant sets. For example, the union of bi-invariant subsets of a binary $G$-space is not necessarily a bi-invariant subset. Difficulties arise also in describing orbits of a binary $G$-space. Given a binary $G$-space $X$, the set $G(x, x)$ is not generally the orbit of $x \in X$, and orbits of points may intersect.

In binary $G$-spaces, orbits are described recursively. As a result, we obtain finitely or infinitely generated orbits. In a distributive binary $G$-space $X$, all orbits are finitely generated. Moreover, for any $x \in X$, the set $G(x, x)$ is bi-invariant; therefore, this set is the orbit of $x$, because it is the minimal bi-invariant subset containing $x$ \cite{Gev2}. This gives rise to the following natural question: Is it true that if all orbits of a binary $G$-space $X$ are of the form $G(x, x)$, then $X$ is a distributive binary $G$-space? This problem was posed in \cite{Gev2}. The present paper contains, in particular, a solution of this problem.

%---------------------------------------------
 
\section{Basic Notions And Notation}

Let $G$ be any topological group, and let $X$ be any topological space.

The space $X$ is called a \emph{$G$-space} if it is equipped with a \emph{continuous action} $\alpha$ of the group $G$, i.e., a continuous mapping $\alpha :G\times X\to X$ satisfying the conditions
$$\alpha(gh, x)=\alpha(g, \alpha(h,x)) \quad \text{and} \quad  \alpha(e,x)=x$$
or, in the notation $\alpha(g,x)=gx$,
$$(gh)x=g(hx) \quad \text{and} \quad  ex=x,$$
where $e$ is the identity element of $G$, for any $g, h \in G$ and $x \in X$.

The set 
$$\text{Ker}\,\alpha=\{g\in G; \ gx=x\}$$
is called the \emph{kernel} of the action $\alpha$. If Ker\,$\alpha=e$, then $\alpha$ is called an \emph{effective action} and $X$, an \emph{effective
G-space}.

\begin{definition} [see \cite{Gev}]\label{def1} 
A \emph{binary action} of a topological group $G$ on a space $X$ is a continuous mapping $\alpha :G\times X^2\to X$ such that, for any $g,h\in G$ and $x_1,x_2 \in X$,
\begin{equation}\label{eq(1)}
gh(x_1,x_2)=g(x_1, h(x_1,x_2)),
\end{equation}
\begin{equation}\label{eq(2)}
e(x_1,x_2)=x_2,
\end{equation}
where $g(x_1,x_2) = \alpha (g, x_1,x_2)$.
A space $X$ equipped with a binary action $\alpha$ of a group $G$, that is, a triple $(G, X, \alpha)$, is called a \emph{binary G-space}.
\end{definition}

Given $A\subset X$ and $g\in G$, we set
$$g(A,A)=\{g(a_1,a_2); \quad a_1,a_2\in A\}.$$

Similarly, for $K\subset G$, we set
$$K(A,A)=\{g(a_1,a_2); \quad g\in K, a_1,a_2\in A\}.$$ 

For brevity, we sometimes write $g(A)$ and $K(A)$ instead of $g(A,A)$ and $K(A,A)$, respectively.

A subset  $A\subset X$ is said to be \emph{bi-invariant} if $G(A,A)=A$. A bi-invariant subset $A$ of $X$ is itself a binary $G$-space; it is called a \emph{binary G-subspace}.

\begin{definition}[see \cite{Gev2}]\label{def-orbita}
The orbit of an element $x$ of a binary $G$-space $X$ is the minimal bi-invariant set $[x]\subset X$ containing $x$.
\end{definition}

Obviously, the set $G(x,x)$ is a subset of the orbit $[x]$:
$$G(x,x)\subset [x].$$

\begin{definition}[see \cite{Gev2}]
A binary $G$-space $X$ is said to be \emph{distributive} if
\begin{equation}\label{eq1-1}
g(h(x,x'), h(x,x''))=h(x,g(x', x'')).
\end{equation}
for any $x,x',x'' \in X$ and $g,h \in G$.
\end{definition}

The class of distributive binary $G$-spaces plays an important role in the theory of binary $G$-spaces. 

Let $H$ be a subgroup of $G$. The set
$$N_G(H)=\{g\in G; \quad g^{-1}Hg=H\}$$
is called the \emph{normalizer} of the subgroup $H$ in the group $G$. The normalizer $N_G(H)$ is the maximal subgroup of $G$ containing $H$ as a normal subgroup. Obviously, if $H$ is a normal subgroup $G$, then $N_G(H) = G$.

The \emph{commutator} of elements $g$ and $h$ of $G$ is the element
$$[g,h]=g^{-1}h^{-1}gh\in G.$$
The subgroup $G'=[G,G]$ generated by all commutators of $G$ is called the \emph{commutator subgroup} of the group $G$. The commutator subgroup $G'$ is a normal subgroup of $G$. The commutator subgroup $G'$ is trivial if and only if the group $G$ is commutative.

These definitions, as well as all other definitions, notions and results of group theory and the theory of continuous transformation groups used in the paper without reference, can be found in \cite{Br} and \cite{Kurosh}.

% %----------------------------------------

\section{Bi-Invariant Subsets Of A Binary $G$-Space }

\begin{proposition}\label{prop-0}
Any intersection of bi-invariant subsets of a binary $G$-space is a bi-invariant subset.
\end{proposition}

\begin{proof}
Suppose given bi-invariant subsets $A$ and $B$ of a binary $G$-space $X$. Let us prove that $G(A\cap B)=A\cap B$. Since $A\cap B\subset G(A\cap B)$, it suffices to show that $G(A\cap B)\subset A\cap B$. Indeed, the inclusions $A\cap B \subset A$ and $A\cap B \subset B$ imply
$$G(A\cap B)\subset G(A)=A, \quad G(A\cap B)\subset G(B)=B.$$
Therefore, $G(A\cap B)\subset A\cap B$.
\end{proof}

A union of bi-invariant subsets in a binary $G$-space, in contrast to that in a $G$-space, is not generally bi-invariant (see \cite[Example 3]{Gev2}).

Let $(G, X, \alpha)$ be a binary $G$-space. The $G$-space $(G, X\times X, \widetilde{\alpha})$, on which the action of $G$ is defined by
\begin{equation*}
\widetilde{\alpha}(g, x_1,x_2)=(x_1, \alpha(g,x_1,x_2)),
\end{equation*}
is called the \emph{natural $G$-square} induced by the binary action $\alpha$.

Setting $\widetilde{\alpha}(g, x_1,x_2)=g\cdot (x_1,x_2)$ and $\alpha(g,x_1,x_2)=g(x_1,x_2)$, we rewrite the last formula as
$$g\cdot (x_1,x_2)=(x_1, g(x_1,x_2)).$$

Note that, for any  $g\in G$ and $a,a'\in A$, 
$$g\cdot (a,a')=(a,g(a,a'))\in A\times A$$
if and only if $g(a,a')\in A$. Thus, the following proposition holds.

\begin{proposition}
A subset $A$ of a binary $G$-space $X$ is bi-invariant if and only if the set $A\times A$ is invariant in the natural $G$-square $X\times X$.
\end{proposition}

Let $(X,G,\alpha)$ be a $G$-space. The unary action $\alpha$ generates a binary action $\overline{\alpha}$ of the group $G$ on $X$ by the rule
\begin{equation}\label{eq11}
\overline{\alpha}(g,x_1,x_2)=\alpha(g, x_2), \quad \text{or} \quad g(x_1,x_2)=gx_2,
\end{equation}
for all $g \in G$ and $x_1, x_2\in X$. The action $\overline{\alpha}$ is called the \emph{induced binary action}, and the $G$-space $(X,G, \overline{\alpha})$ is called the \emph{induced binary $G$-space}.

The following simple proposition is valid.

\begin{proposition}\label{prop_1}
Let $(X,G, \alpha)$ be any $G$-space. A set $A\subset X$ is bi-invariant with respect to the induced binary action $\overline{\alpha}$ if and only if it is invariant with respect to the action $\alpha$.
\end{proposition}

In a binary $G$-space $X$, any bi-invariant subset containing a point $x\in X$ contains the whole set $G(x,x)$. There arises the natural question of whether the set $G(x,x)$ is bi-invariant. Example \ref{ex-1} constructed at the end of this section shows that, in the general case, the answer in negative. However, there exists a large class of binary $G$-spaces in which the sets $G(x,x)$ are bi-invariant.

\begin{proposition}
Let $X$ be a distributive binary $G$-space. Then, for any elements $x, x' \in X$, 
\[
G(G(x,x),G(x,x'))=G(x,x').
\]
\end{proposition}

\begin{proof}
Indeed, in view of the distributivity of the binary action, we have
\begin{multline*}
g(h(x,x),k(x,x'))=g(h(x,x),h(x, h^{-1}k(x,x')))= \\
=h(x, g(x,h^{-1}k(x,x')))=h(x, g h^{-1}k(x,x'))= \\
=hg h^{-1}k(x, x') \in G(x,x') 
\end{multline*}
for any $g,h,k \in G$.
\end{proof}

A direct consequence of this proposition is the following theorem.

\begin{theorem}[\cite{Gev2}]\label{th-9}
Let $X$ be a distributive binary $G$-space. Then the set $G(x,x)$ is bi-invariant for any $x\in X$.
\end{theorem}

There arises the natural question of whether the converse of this theorem is true. This problem was posed in \cite{Gev2}.

\begin{problem}{\cite{Gev2}}\label{prob-1}
Let $X$ be a binary $G$-space such that the set $G(x,x)\subset X$ is bi-invariant for any $x\in X$. Is it true that $X$ is a distributive binary $G$-space?
\end{problem}

The rest of this section is devoted to the solution of this problem.

\begin{example}\label{ex-(H,G)}
Let $G$ be a topological group, and let $H$ be its subgroup. The continuous mapping  $\alpha : H\times G^2 \to G$ defined by
\begin{equation}\label{eq-action}
\alpha (h,x_1,x_2) = x_1^{-1}hx_1x_2, \quad \text{or} \quad h(x_1,x_2) = x_1^{-1}hx_1x_2,
\end{equation}
for any $h\in H$ and $x_1, x_2 \in G$ determines a binary action of the subgroup $H$ on the group $G$. Indeed, we have $e(x_1,x_2)=x_2$ and
\begin{multline*}
hh'(x_1,x_2) = x_1^{-1}hh'x_1x_2 = x_1^{-1}hx_1x_1^{-1}h'x_1x_2 = \\
=h(x_1, x_1^{-1}h'x_1x_2) = h(x_1, h'(x_1,x_2)).
\end{multline*}

We denote the binary $H$-space thus obtained by $(H,G,\alpha)$.
\end{example}

\begin{proposition}\label{prop-1}
Suppose given a binary $H$-space $(H,G,\alpha)$. A subset $H(x,x)\subset G$, $x\in G$, is bi-invariant if and only if $x^{-1}Hx\subset N_G(H)$, where $N_G(H)$ is the normalizer of the subgroup $H$ in the group $G$.
\end{proposition}

\begin{proof}
The bi-invariance of $H(x,x)$, $x\in G$, means that, for any  $h,h_1$, $h_2\in H$, there exists an element $\widetilde{h}\in H$ such that
\[
h(h_1(x,x), h_2(x,x))=\widetilde{h}(x,x).
\]

By virtue of \eqref{eq-action}, this implies
$$
h(x^{-1}h_1xx, x^{-1}h_2xx)=x^{-1}\widetilde{h}xx, 
$$
$$
x^{-1}x^{-1}h_1^{-1}xhx^{-1}h_1xx x^{-1}h_2xx=x^{-1}\widetilde{h}xx,
$$
$$
x^{-1}h_1^{-1}xhx^{-1}h_1xh_2=\widetilde{h},
$$
$$
(x^{-1}h_1x)^{-1}h(x^{-1}h_1x)=\widetilde{h}h_2^{-1}\in H,
$$
which is equivalent to $x^{-1}Hx\subset N_G(H)$.
\end{proof}

\begin{corollary}
If $H$ is a normal subgroup of a group $G$, then $H(x,x)$ is bi-invariant for any $x\in G$.
\end{corollary}

 \begin{theorem}
 Let $(X,G, \alpha)$ be a $G$-space. The induced binary $G$-space $(X,G, \overline{\alpha})$ is distributive if and only if the commutator subgroup $G'$ of the group $G$ is a subgroup of the kernel $\rm{Ker}\,\alpha$ of the action  $\alpha$.
 \end{theorem}
 
\begin{proof}
Suppose that the induced binary $G$-space $X$ is distributive. Then, for any $g,h\in G$ and $x\in X$, we have
\begin{multline*}
g(h(x,x), h(x,x)) = h(x, g(x,x)) \  \Longrightarrow \\
\Longrightarrow (gh)x = (hg)x \  \Longrightarrow \ (g^{-1}h^{-1}gh) x =x,
\end{multline*}
i.e., $g^{-1}h^{-1}gh = [g,h]\in \rm{Ker}\, \alpha $.

Now suppose that $g^{-1}h^{-1}gh = [g,h]\in \rm{Ker}\,\alpha $ for any $g,h\in G$. Let us verify that the induced binary
action is distributive:
\[
g(h(x,x'), h(x,x'')) =  (gh)x'' = (hg)x'' = h(x, g(x',x'')).
\]
\end{proof}

\begin{corollary}\label{cor-1}
If $(X,G, \alpha)$ is an effective $G$-space, then the induced binary $G$-space $(X,G, \overline{\alpha})$ is distributive if and only if $G$ is an Abelian group.
\end{corollary}
 
\emph{Solution of Problem \ref{prob-1}.} Let $G$ be any non-Abelian group, and let $X$ be an effective $G$-space. Consider the induced binary $G$-space $(X,G, \overline{\alpha})$. According to Proposition \ref{prop_1}, the sets $G(x,x)\subset X$ are bi-invariant for all $x\in X$. However, the binary $G$-space $(X,G,\overline{\alpha})$ is not distributive by virtue of Corollary \ref{cor-1}. Thus, Problem \ref{prob-1} has a negative solution.

Now we construct the promised example of a non-bi-invariant set of the form $G(x,x)$.

\begin{example}\label{ex-1}
Let $G=GL(2,\mathbf{R})$. Consider the set
$$H=\left\{ 
\begin{bmatrix}
1 & h\\
0 & 1
\end{bmatrix}, \  h\in \mathbf{R} \right\}.
$$
This is a subgroup of the group $GL(2,\mathbf{R})$ and, therefore, formula \eqref{eq-action} defines a binary action of $H$ on $GL(2,\mathbf{R})$.

Let us show that, for 
$
x=\begin{bmatrix}
0 & 1\\
1 & 0
\end{bmatrix}
$,
the subset $H(x,x)$ of $GL(2,\mathbf{R})$ is not bi-invariant. Consider the matrix 
$h=
\begin{bmatrix}
1 & 1\\
0 & 1
\end{bmatrix}
\in H.
$
Note that the matrix
$$x^{-1}hx=
\begin{bmatrix}
0 & 1\\
1 & 0
\end{bmatrix} 
\begin{bmatrix}
1 & 1\\
0 & 1
\end{bmatrix} 
\begin{bmatrix}
0 & 1\\
1 & 0
\end{bmatrix} = 
\begin{bmatrix}
1 & 0\\
1 & 1
\end{bmatrix}
$$
does not belong to the normalizer $N_G(H)$. Indeed, we have
$$(x^{-1}hx)^{-1}h(x^{-1}hx)= 
\begin{bmatrix}
1 & 0\\
-1 & 1
\end{bmatrix} 
\begin{bmatrix}
1 & 1\\
0 & 1
\end{bmatrix} 
\begin{bmatrix}
1 & 0\\
1 & 1
\end{bmatrix} = 
\begin{bmatrix}
2 & 1\\
-1 & 0
\end{bmatrix}
\notin H. \vspace{6pt}
$$

Therefore, by virtue of Proposition \ref{prop-1}, H $H(x,x)$ is not a bi-invariant subset of the group $GL(2,\mathbf{R})$.
\end{example}

% %-----------------------------------

\section{Orbits Of A Binary $G$-Space}

In binary $G$-spaces, unlike in $G$-spaces, orbits may intersect.

\begin{example}
Let $GL(2,\mathbf{R})$ be the topological group of nonsingular square matrices of order 2. The
set $H=\{e,h\}$, where 
$$
e=
\begin{bmatrix}
1 &  0  \\
0 & 1 
\end{bmatrix},  \quad 
h=\begin{bmatrix}
0 &  1 \\
1 & 0
\end{bmatrix},
$$
is a subgroup of $GL(2,\mathbf{R})$, because $h^2=e$. Therefore, formula \eqref{eq-action} defines a binary action of $H$ on $GL(2,\mathbf{R})$ (see Example \ref{ex-(H,G)}). Obviously, the orbit of the element $h$ coincides with the subgroup $H$: $[h]=\{e,h\}$. Now consider the matrix 
$x=
\begin{bmatrix}
0 &  -1  \\
1 & -1 
\end{bmatrix}\notin [h]
$. 
Let us prove that the orbit $[x]$ intersects the orbit $[h]$. Indeed,
$$
h(x,x)=x^{-1}hxx=\begin{bmatrix}
-1 &  1  \\
-1 & 0 
\end{bmatrix}
\begin{bmatrix}
0 &  1 \\
1 & 0
\end{bmatrix}
\begin{bmatrix}
0 &  -1  \\
1 & -1 
\end{bmatrix}
\begin{bmatrix}
0 &  -1  \\
1 & -1 
\end{bmatrix}=
\begin{bmatrix}
0 &  1 \\
1 & 0
\end{bmatrix}=
h. 
$$

This example shows that the orbit $[x]$ of a point $x$ in a binary $G$-space $X$ does not generally coincide with the orbits of all points of $[x]$. It may contain smaller orbits. However, there exists a large class of binary $G$-spaces with disjoint orbits.
\end{example}

\begin{proposition}
Two orbits of a distributive binary $G$-space $X$ either are disjoint or coincide.
\end{proposition}

\begin{proof}
Proof. Let $x$ be any element of a distributive binary $G$-space $X$. By Theorem \ref{th-9}, the set   $G(x,x)$, $x\in X$,
is bi-invariant; therefore, it is the orbit of the point $x$: $[x] = G(x, x)$.

First, we prove that $[x]$ is the orbit of each of its points, i.e., contains no proper bi-invariant subsets. Take any element $g_0(x,x)\in [x]$. Let us show that $[g_0(x,x)]=[x]$. Obviously, $[g_0(x,x)]\subset [x]$. Therefore, it suffices to check that $[x] \subset [g_0(x,x)]$, i.e., any element $g(x,x)$ of the orbit $[x]$ is also an element of the orbit $[g_0(x,x)]$: 
$$g(x,x) = g'(g_0(x,x), g_0(x,x))$$ 
for some $g'\in G$. Indeed, taking into account the distributivity of the binary action, we obtain
\begin{multline*} 
g(x,x)= g_0g_0^{-1}g(x,x) = g_0(x, g_0^{-1}g(x,x))  =  \\
=g_0^{-1}g(g_0(x,x), g_0(x,x)) = g'(g_0(x,x), g_0(x,x)),
\end{multline*}
where $g'=g_0^{-1}g$.

Now let $[x]$ and $[x']$ be any orbits of a distributive binary $G$-space $X$. Suppose that these orbits intersect, i.e., there exists a point $\widetilde{x}\in [x]\cap [x']$. According to what was proved above, we have $[x]=[\widetilde{x}]$ and $[x']=[\widetilde{x}]$. Therefore, the orbits $[x]$ and $[x']$ coincide.
\end{proof}

\begin{example}
Let $G$ be a topological group, and let $H$ be a subgroup of $G$. The continuous mapping $\alpha : H\times G^2\to G$ defined by
\begin{equation}\label{eq-H on G}
\alpha(h,x,y) = xhx^{-1}y \quad \text{or} \quad h(x,y) = xhx^{-1}y,
\end{equation}
$h\in H$ and $x,y\in G$ are any elements, is a binary action of $H$ on $G$. Indeed,
$$h
\widetilde{h}(x,y)=xh\widetilde{h}x^{-1}y=xhx^{-1}x\widetilde{h}x^{-1}y=h(x, x\widetilde{h}x^{-1}y)=h(x, \widetilde{h}(x,y)),
$$
$$
e(x,y)=xex^{-1}y=y
$$
for all $h, \widetilde{h}\in H$ and $x,y \in G$.
In contrast to \eqref{eq-action}, the binary action \eqref{eq-H on G} is distributive:
\begin{multline*}
h(\widetilde{h}(x,y), \widetilde{h}(x,z))=h(x\widetilde{h}x^{-1}y, x\widetilde{h}x^{-1}z)= \\
=x\widetilde{h}x^{-1}yhy^{-1}x\widetilde{h}^{-1}x^{-1}x\widetilde{h}x^{-1}z=x\widetilde{h}x^{-1}yhy^{-1}z= \\
=\widetilde{h}(x,yhy^{-1}z)=\widetilde{h}(x,h(y, z)).
\end{multline*}
for all $h,\widetilde{h}\in H$ and $x,y,z\in G$.

Thus, the orbits of the binary $H$-space under consideration are the set of the form $H(x,x)$. Since $H(x,x) = xHx^{-1}x = xH$, it follows that the orbits are the left cosets of $H$ in $G$.
\end{example}

To describe the orbit of a point $x$ of any binary $G$-space $X$, we recursively define sets $G^{n}(x,x)$, $n\in N$,
by
\[
G^1(x,x) = G(x,x),  \quad \ldots , \quad G^n(x,x) = G(G^{n-1}(x,x), G^{n-1}(x,x)).
\]
It is easy to see that

(1) $x\in G^1(x,x)\subset G^2(x,x)\subset \ldots \subset G^n(x,x) \subset ...$;

(2) if $G^n(x,x)$ is bi-invariant for some $n\in N$, then $G^k(x,x)=G^n(x,x)$ for all positive integers $k>n$ and, therefore, $\bigcup\limits_{i=1}^{\infty} G^i(x,x) = G^n(x,x)$;

(3) the set   $\bigcup\limits_{i=1}^{\infty} G^i(x,x)$ is bi-invariant in the binary $G$-space $X$.

These assertions directly imply the following proposition.

\begin{proposition}
For any element $x$ of a binary $G$-space $X$,
\[
[x]= \bigcup\limits_{i=1}^{\infty} G^i(x,x),
\] 
where $[x]$ is the orbit of $x$.
\end{proposition}

\begin{definition}
The orbit of a point $x$ of a binary $G$-space $X$ is said to be \emph{finitely generated} if $[x]=G^n(x,x)$ for some $n\in N$. Otherwise, the orbit is said to be \emph{infinitely generated}.
\end{definition}

In the examples of binary $G$-spaces constructed previously, including those of distributive binary $G$-spaces, all orbits are finitely generated. However, the following theorem is valid.

\begin{theorem}
There exist binary G-spaces with infinitely generated orbits.
\end{theorem}

\begin{proof}
Let $G$ be an infinite topological group. Suppose that there exist elements $h,x\in G$ satisfying the following conditions:

(1) the elements $h$ and $x$ are of order 2: $h^2=x^2=e$, where $e$ is the identity element of $G$; 

(2) the element $xh$ is of infinite order.

For example, let $G=GL(2,\mathbf{R})$ be the topological group of nonsingular square matrices of order 2. Consider 
$$
h=
\begin{bmatrix}
1 &  0  \\
0 & -1 
\end{bmatrix},  \quad 
x=\begin{bmatrix}
-1 &  0 \\
1 & 1
\end{bmatrix}.
$$
It is easy to verify that the matrices $h$ and $x$ satisfy conditions (1) and (2).

Consider the subgroup $H=\{e,h\}$ of $G$. It acts binarily on $G$ by the rule \eqref{eq-action} (see Example \ref{ex-(H,G)}). Let us prove that, in this binary $H$-space, the orbit of the element $x$ is infinitely generated. To this end, it suffices to show that $H^{n+1}(x,x)\neq H^n(x,x)$ for all positive integers $n$.

It immediately follows from the definition of the binary action \eqref{eq-action} and condition (1) that any element of $H^n(x,x)$ different from $x$ has one of the following forms:
$${\rm(i)} \ \  (xh)^i, \qquad {\rm(ii)} \ \ (xh)^ix, \quad {\rm(iii)} \ \ h(xh)^i, \qquad {\rm(iv)} \ \ h(xh)^ix,$$
where $i$ is a positive integer. For example,
\begin{equation*}
    e(x,x)=x,
\end{equation*}
\begin{equation*}
    h(x,x)=x^{-1}hxx=xh,
\end{equation*}
\begin{equation*}
    h(h(x,x),x)=(xh)^{-1}h(xh)x=hxhxhx=h(xh)^2x,
\end{equation*}
\begin{equation*}
    h(x,h(x,x))=x^{-1}hxh(x,x)=xhxxh=x,
\end{equation*}
\begin{equation*}
    h(h(x,x),h(x,x))=h(xh,xh)=(xh)^{-1}h(xh)xh=hxhxhxh=h(xh)^3.
\end{equation*}
Therefore,
$$H^1(x,x)=\{x,xh\}, \quad  H^2(x,x)=\{x,xh, h(xh)^2x, h(xh)^3\}.$$

Since $H^{n}(x,x)$ is finite, there exists an element $y\in H^{n}(x,x)$ which contains $xh$ to the highest power $k$. Let us prove that at least one of the elements $h(x,y)$ and $h(xh,y)$ belongs to $H^{n+1}(x,x)$ and does not belong to $H^{n}(x,x)$. Consider the possible cases.

Case 1. Suppose that $y$ has the form (i): $y=(xh)^k$. Then
\begin{multline*}
h(xh,y)=h(xh,(xh)^k)= (xh)^{-1}h(xh)(xh)^k= \\
=hxhxh(xh)^k=h(xh)^{k+2}.
\end{multline*}

Case 2. If $y$ has the form $y=(xh)^kx$, then
\begin{multline*}
    h(xh,y)=h(xh,(xh)^kx)= (xh)^{-1}h(xh)(xh)^kx= \\
    = hxhxh(xh)^kx=h(xh)^{k+2}x.
\end{multline*}

Case 3. If $y$ has the form  $y=h(xh)^k$, then
$$h(x,y)=h(x,h(xh)^k)= x^{-1}hxh(xh)^k= xhxh(xh)^k=(xh)^{k+2}.$$ 

Case 4. If $y$ has the form $y=h(xh)^kx$, then
$$h(x,y)=h(x,h(xh)^kx)= x^{-1}hxh(xh)^kx= xhxh(xh)^kx=(xh)^{k+2}x.$$ 

In all cases, the element specified above belongs to the set $H^{n+1}(x,x)$ but does not belong to $H^{n}(x,x)$, because $xh$ has infinite order and its power is higher than $k$. Therefore,
$$H^{n+1}(x,x)\neq H^n(x,x).$$ 

This completes the proof of the theorem.
\end{proof}

\bibliographystyle{plain}
\bibliography{gevorgyan_nazaryan}

\begin{thebibliography}{1}

\bibitem{Br}
G.~E. Bredon.
\newblock {\em Introduction to Compact Transformation Groups}.
\newblock New York, 1972.

\bibitem{Gev}
P.~S. Gevorgyan.
\newblock On binary ${G}$-spaces.
\newblock {\em Mathematical Notes}, 96(4):600--602, 2014.

\bibitem{Gev2}
P.~S. Gevorgyan.
\newblock Groups of binary operations and binary ${G}$-spaces.
\newblock {\em Topology and its Applications}, 201:18--28, 2016.

\bibitem{Gev-izv}
P.~S. Gevorgyan.
\newblock Groups of invertible binary operations of a topological space.
\newblock {\em J.Contemp.Math.Anal.}, 53(1):16--20, 2018.

\bibitem{Kurosh}
A.~G. Kurosh.
\newblock {\em Group Theory}.
\newblock Nauka, Moscow, 1967 [inRussian], 1967.

\end{thebibliography}

\end{document}